\documentclass[12pt,letterpaper]{amsart}

\usepackage{amssymb,ccfonts,enumerate,dsfont}
\usepackage{fourier}
\usepackage{Baskervaldx} 
 
\usepackage[text={7.1in,8.6in},centering]{geometry} 

\newcommand{\inv}{^{\raisebox{.2ex}{$\scriptscriptstyle-1$}}}

\newcommand{\la}{\mathcal{Q}}
\newcommand{\id}{\mathcal{I}_\la}

\newcommand{\idp}{\mathcal{I}^+_\la}

\newcommand{\ua}{^{\uparrow}}
\newcommand{\spc}{\mathrm{Spec}_\la}
\newcommand{\mx}{\mathrm{Max}_\la}

\newcommand{\pr}{\mbox{{\small{$\&$}}}}

\usepackage[dvipsnames]{xcolor}   
\usepackage{xparse}
\usepackage{xr-hyper}
\usepackage[linktocpage=true,colorlinks=true,hyperindex,citecolor=teal,linkcolor=Sepia]{hyperref} 

\newtheoremstyle{theoremdd}% name of the style to be used
{15pt}% measure of space to leave above the theorem. E.g.: 3pt
{10pt}% measure of space to leave below the theorem. E.g.: 3pt
{\itshape}% name of font to use in the body of the theorem
{0pt}% measure of space to indent
{\scshape}% name of head font
{. ---}% punctuation between head and body
{ }% space after theorem head; " " = normal interword space
{\thmname{#1}\thmnumber{ #2}\textnormal{\thmnote{ (#3)}}}
\theoremstyle{theoremdd}
\newtheorem{theorem}{Theorem}[section]
\newtheorem{proposition}[theorem]{Proposition}
\newtheorem{lemma}[theorem]{Lemma}
\newtheorem{corollary}[theorem]{Corollary}
\newtheoremstyle{thmdd}% name of the style to be used
{15pt}% measure of space to leave above the theorem. E.g.: 3pt
{10pt}% measure of space to leave below the theorem. E.g.: 3pt
{}% name of font to use in the body of the theorem
{0pt}% measure of space to indent
{\scshape}% name of head font
{. ---}% punctuation between head and body
{ }% space after theorem head; " " = normal interword space
{\thmname{#1}\thmnumber{ #2}\textnormal{\thmnote{ (#3)}}}
\theoremstyle{thmdd}

\newtheorem{remark}[theorem]{Remark}

\numberwithin{equation}{section}
 
\begin{document}

\author{Amartya Goswami}
\address{\scriptsize{[1] Department of Mathematics and Applied Mathematics, University of Johannesburg, South Africa;
[2] National Institute for Theoretical and Computational Sciences (NITheCS), South Africa.}}
\email{\scriptsize{agoswami@uj.ac.za}}

\title{On ideals in quantales, II}

\date{}

%quantales;

%Spectra in general topology


\maketitle

\begin{abstract}
\scriptsize{As a natural extension of the ongoing development of a theory of ideals in commutative quantales with an identity element, this article aims to study into the analysis of certain topological properties exhibited by distinguished classes of ideals. These ideals are equipped with quantale topologies. The primary objectives encompass characterizing quantale spaces exhibiting sobriety, examining several conditions pertaining to quasi-compactness, and demonstrating that quantale spaces comprised of proper ideals adhere to the spectral properties as defined by Hochster. We introduce the notion of  a strongly disconnected spaces and show that for a quantale with zero Jacobson radical, strongly disconnected quantale spaces containing all maximal ideals of the quantale imply existence of non-trivial idempotent  elements in the quantale. Additionally, a sufficient criterion for establishing the connectedness of a quantale space is presented. Finally, we discuss on continuous maps between quantale spaces.}
\end{abstract}

\smallskip 
\section{Introduction}
\smallskip 

The concept of introducing topological structures to different classes of ideals traces its origins back to \cite{Sto37}. In that work, a specific topology referred to as the Stone topology or hull-kernel topology was established for the set of prime ideals in a Boolean ring.

Subsequently, in \cite{Jac45}, a similar hull-kernel topology, recognized as the Jacobson topology, was imposed on the collection of primitive ideals in a ring. Notably, \cite{McC48} noted the applicability of this same topology to the set of generalized prime ideals as defined therein.

In the context of affine schemes, as explored in \cite{Gro60}, the hull-kernel topology (known as the Zariski topology) was considered for the set of prime ideals in a commutative ring. Building upon this, \cite{Olb16} demonstrated the utilization of the Zariski topology in establishing that every ideal can be uniquely represented in terms of irreducible ideals, eliminating redundancy.

Further investigations ensued in the field. \cite{HJ65} studied into the hull-kernel topology's behaviour concerning minimal prime ideals within a commutative ring. Meanwhile, \cite{Azi08} embarked on a study of strongly irreducible ideals in a commutative ring, specifically within the context of the Zariski topology.
Exploration of the topological properties exhibited by distinct classes of ideals in a commutative ring extended to \cite{DG22} and \cite{FGS23}.

Beyond the realm of rings, investigations into the imposition of topologies on diverse classes of ideals or substructures extend to various algebraic structures. For instance, within the domain of $R$-modules, \cite{FFS16} examines topologies applied to ideals and proper ideals, while \cite{FS20} delves into primary ideals and principal ideals with constructible topologies.

Turning to semirings, topologies applied to prime, maximal, and strongly irreducible ideals are explored in \cite{Is\'{e}56, Is\'{e}56', IM56, Gol99}. In the specialized context of $(m,n)$-semirings, \cite{H al.18} investigates the structural aspects of $n$-ary prime $k$-ideals, $n$-ary prime full $k$-ideals, $n$-ary prime ideals, maximal ideals, and strongly irreducible ideals.

In the domain of semigroups, \cite{Gos23''} delves into the application of topologies to primitive ideals, while \cite{Gos22} undertakes a similar exploration for various types of subgroups within a group.

A comprehensive exploration of different topologies related to prime spectra, including variations like minimal prime, coprime, and fully prime spectra, is available in \cite{A11, A15, A10, AF14, AH19, AKF16, AP17, BH08, MMS97, MS06, Tek09, UTS15}. Moreover, the study of prime spectra within multiplicative lattices is addressed in \cite{Fac23, FFJ22}.

The objective of this paper is to extend these investigations into the realm of lattice theory, specifically within the context of commutative unital quantales. This expansion is not only reasonable but also well-founded, as the collection of ideals in a commutative ring inherently forms a quantale. Furthermore, by considering intersection as the multiplication operation, we uncover additional instances of quantales that emerge as substructures of the aforementioned algebraic entities.

For a comprehensive exploration of the broader theory of quantales (also known as a special type of multiplicative lattices), we direct our readers to references such as \cite{Ros90, Mul86, And74, Dil36, Dil62, KP08, War37, WD39}. The concept of an ideal within a quantale was introduced in \cite{Gos23'}, where algebraic properties of specific distinguished classes of quantale ideals were examined.

Within the scope of this paper, we introduce the adaptation of a specific topology, termed the ``quantale topology,'' which can be applied to any class of ideals in a quantale. Our primary goal is to study various topological aspects intrinsic to these corresponding spaces.

Let us briefly outline the paper's structure. In Section \ref{prlm}, we revisit crucial definitions and pertinent facts concerning quantale ideals. The groundwork for imposing quantale topologies on various classes of ideals is laid out in Section \ref{qtt}. We elucidate our choice of subbasic closed sets as the fundamental constituents for constructing our topology.
Section \ref{qct} deals with the examination of quasi-compactness of quantale spaces, offering a range of characterizations. The discussion then shifts to lower separation axioms for quantale spaces, encompassing $T_0$, $T_1$, and sobriety in Section \ref{spax}.
Demonstrating the spectrality of the quantale space of proper ideals is also the focus of study in this section. Finally, in Section \ref{cac}, we talk about the concept of continuity among quantale spaces and introduce a distinctive form of connectedness that correlates with the presence of idempotent elements in a quantale.

\section{Preliminaries}\label{prlm}
\smallskip 

Let us now revisit fundamental definitions pertaining to ideals of quantales. For a detailed study of these concepts and related results, we refer our readers to \cite{Gos23'}.
Recall that an \emph{unital commutative quantale} is a complete lattice $(\la, \preccurlyeq, \bot, \top)$ endowed with a binary operation $\pr $ on $\la$, satisfying the following axioms:
\begin{enumerate}
		
\item\label{mla} $x\pr  (y\pr  z)=(x\pr  y)\pr  z$,
		
\item\label{mlc} $x\pr  y = y\pr  x$,
		
\item\label{jdp} $x\pr  \left(\bigvee\limits_{\lambda \in \Lambda} y_{\lambda}\right)=\bigvee\limits_{\lambda \in \Lambda}\left(x\pr  y_{\lambda}\right)$, 
		
\item \label{mid} $x\pr  \top=x$,
\end{enumerate}
for all $x,$ $y,$ $y_{\lambda}$, $z\in \la$, and for all $\lambda \in \Lambda$, where $\Lambda$ is an index set.
Throughout this paper, unless otherwise stated, by a ``quantale,'' we shall always refer to a commutative quantale with the top element  $\top$ as the identity with respect to $\pr$. We shall use the notation $x^n$ to denote $x\pr  \cdots \pr  x$ (repeated $n$ times).
A nonempty subset $I$ in $\la$ is called an \emph{ideal}, if for all $x,$ $y,$ $l\in \la$, the following two axioms hold:
\begin{enumerate}
		
\item\label{cuj} $x,$ $ y \in I$ implies that $x\vee y\in I$,  and 
		
\item\label{cul} $x\in I$ and $l\preccurlyeq x$ implies that $l\in I$.
\end{enumerate}
We shall denote the set of all ideals in $\la$ by $\id$, and by $\theta$, the zero ideal in $\la$. 
Let $\la$ be a quantale.
The \emph{meet} $\bigwedge_{\lambda \in \Lambda} I_{\lambda}$ of ideals $\{I\}_{\lambda\in \Lambda}$  in $\la$ is given by their intersection.
The \emph{join} of ideals $\{I\}_{\lambda\in \Lambda}$  of $\la$ is defined by
\begin{equation*}
\bigvee_{\lambda \in \Lambda} I_{\lambda}=\left\{l\in \la\mid l\preccurlyeq \bigvee_{\text{finite}} x_{\lambda}\;\text{for}\; x_{\lambda}\in I_{\lambda}\right\}.
\end{equation*}
The \emph{product} of two ideals $I$ and $J$ in $\la$ is defined as
\begin{equation*}\label{prij}
I\pr J=\left\{l\in \la\mid l\preccurlyeq \bigvee_{\text{finite}} \left\{x\pr y\mid x\in I, y\in J\right\}\right\}.
\end{equation*}
If $S$ is a nonempty subset in a $\la$, then the \emph{ideal generated by} $S$ is defined by \[\langle S\rangle=\left\{x\in \la \mid x\preccurlyeq \bigvee_{i=1}^na_i,\;\text{for some}\;n\in \mathds{N}\;\text{and}\; a_i\in \la \pr S\right\}.\] 
Let $\la$ be a quantale.
An ideal $I$ of $\la$ is called \emph{proper} if $I\neq \la$, and we shall denote the set of all proper ideals in $\la$ by $\idp$.
A proper ideal $M$ of $\la$ is called \emph{maximal} if there is no other proper ideal $I$ of $\la$ properly containing $M$, and we denote the set of all maximal ideals in $\la$ by $\mathrm{Max}_\la$. A quantale $\la$ with exactly one maximal ideal is called \emph{local}.
An ideal $I$ of $\la$ is called \emph{minimal} if $I\neq 0$ and $I$ properly contains no other nonzero ideals in $\la$.

A proper ideal $P$ of  $\la$ is called \emph{prime} if $x\pr y\in P$ implies $x\in P$ or $y\in P$ for all $x,$ $y\in \la$. By $\spc$, we denote the set of all prime ideals in $\la$. An ideal $P$ of $\la$ is called \emph{minimal prime} if $P$ is both a minimal ideal and a prime ideal. A quantale
$\la$ is called \emph{Noetherian} if every ascending chain of ideals in $\la$ is eventually stationary. An element $e$ in a quantale $\la$ is called \emph{idempotent} if $e^2=e$.
The \emph{Jacobson radical} $\mathcal{J}(\la)$ of a quantale $\la$ is defined as the intersection of all maximal ideals in $\la$.

A proper ideal $P$ in a quantale $\la$ is called \emph{primary} if for $x,$ $y\in \la$ and $x\pr y\in I$ imply that $x\in I$ or $y^n\in I$ for some  $n\in \mathds{N}$. An ideal $I$ of $\la$ is called \emph{irreducible} if  $J\wedge  J'= I$ implies that either $J= I$ or $J'= I$ for all $J,$ $J'\in \id$.
An ideal $I$ of $\la$ is called \emph{strongly irreducible} if  $J\wedge  J'\preccurlyeq I$ implies that either $J\preccurlyeq  I$ or $J'\preccurlyeq I$ for all $J,$ $ J'\in \id$.
By $\Sigma_{\la}$, we shall mean either $\idp$ or any subclass of ideals of it. All the classes of ideals mentioned above are examples of $\Sigma_{\la}$.

\smallskip  

\section{Quantale topology}\label{qtt}
\smallskip 

Consider a quantale $\la$ and let $\Sigma_{\la}$ be a class of ideals in $\la$. Our objective is to define a topology on $\Sigma_{\la}$. We aim to achieve a topology that coincides with the Stone topology when $\Sigma_{\la}$ consists of maximal ideals or with the Zariski topology when $\Sigma_{\la}$ consists of prime ideals in $\la$, and so on. In case of rings, however, it is well-known that neither the Zariski topology, the Stone topology, nor any hull-kernel topology can be imposed on an arbitrary class of ideals.

We will now demonstrate that a similar limitation applies to quantales. The following theorem characterizes the ``largest'' class of ideals in a quantale on which a hull-kernel topology can be imposed. This result, which was initially established in \cite[\textsection 2.2, p.11]{McK53} for rings, is adapted to the context of quantales. 
For a class $\Sigma_{\la}$ of ideals of $\la$, we consider the collection \[\mathcal{B}_{\Sigma_{\la}}:=\left\{X\ua\mid X\preccurlyeq \id\right\}\] of subsets of $\Sigma_{\la}$ given by:
\begin{equation}\label{sbcs}
X\ua =\begin{cases}
\left\{J\in \Sigma_{\la} \mid \mathcal{D}_X:=\wedge X\preccurlyeq J\right\}, & X\neq \emptyset;\\
\emptyset, & X=\emptyset.	
\end{cases}
\end{equation}  
Whenever $X=\{I\},$ we shall write $I^{\uparrow}$ for $\{I\}^{\uparrow}.$

\begin{theorem}\label{hkt}
The collection $\mathcal{B}_{\Sigma_{\la}}$ induces a closed-set topology on  $\Sigma_{\la}$  if and only if the class  $\Sigma_{\la}$ of ideals satisfies the following property:    
\begin{equation}\label{hkp}
I\wedge  J\preccurlyeq S\;\; \text{implies}\;\; I\preccurlyeq S\;\; \text{or}\;\; J\preccurlyeq S,	
\end{equation}
for all $I,$ $J\in \id$ and for all $S\in  \Sigma_{\la}.$  
\end{theorem}  

\begin{proof}  
We shall show that $(-)\ua$ is a Kuratowski closure operator on $\Sigma_{\la}$ if and only if the class $\Sigma_{\la}$ of ideals satisfies the condition (\ref{hkp}). The facts that $\emptyset\ua=\emptyset$ and $X\preccurlyeq X\ua$ follow from (\ref{hkp}), whereas the condition   $(X\ua)\ua=X\ua$ follows from the monotone property of $(-)\ua$. Finally, what remains is to show that $(X\vee Y)\ua=X\ua \vee Y\ua$ holds for any $X,$ $Y\preccurlyeq \id$ if and only if $\Sigma_{\la}$ satisfies (\ref{hkp}). The fact that $X\ua \vee Y\ua\preccurlyeq(X\vee Y)\ua$ follows from monotonicity of $(-)\ua$ and from the properties proven above. Suppose $J\in (X\vee Y)\ua.$ Then $(\mathcal{D}_X)\vee (\mathcal{D}_Y)\preccurlyeq J,$ and hence, $J\in X\ua\vee Y\ua$ if and only if $\Sigma_{\la}$ satisfies (\ref{hkp}).
\end{proof}  

\begin{remark}
It is worth noting that the class of ideals in a quantale that satisfies property (\ref{hkp}) defined earlier corresponds exactly to the class of strongly irreducible ideals. Moreover, it can be easily observed that the classes of maximal, prime, and minimal prime ideals in a quantale also satisfy property (\ref{hkp}). Therefore, these classes serve as suitable candidates upon which a hull-kernel topology can be imposed.
\end{remark}

Whenever $(-)\ua$ is a Kuratowski closure operator on a class $\Sigma_{\la}$ of ideals in $\la$, the collection $\mathcal{B}_{\Sigma_{\la}}$ of subsets of $\Sigma_{\la}$ are exactly the closed subsets of the hull-kernel  topology induced on $\Sigma_{\la}$. Therefore, $\mathcal{B}_{\Sigma_{\la}}$ is a closed basis of that topology. In fact, we can obtain even more, as demonstrated by the following theorem.

\begin{theorem}\label{kclcb} 
If $\la$ is a quantale, then  $(-)\ua$ is a Kuratowski closure operator on a class of ideals $\Sigma_{\la}$ if and only if $\mathcal{B}_{\Sigma_{\la}}$ is a closed basis of the corresponding hull-kernel topology. 
\end{theorem}

\begin{proof}      
If $(-)\ua$ is a Kuratowski closure operator on $\Sigma_{\la}$, then obviously $\mathcal{B}_{\Sigma_{\la}}$ is a closed basis of the corresponding hull-kernel topology. 
To prove the converse, we shall use the following fact: if $\mathcal{B}$ is a closed basis on a topological space $X$, then the  closure operator induced by the topology  is expressible in terms of $\mathcal{B}$, in the sense that 
\[
\mathcal{C}(T)={\wedge }\{B\in \mathcal{B}\mid T\preccurlyeq B\}, \] for any $T\preccurlyeq  X$.
Suppose that $\mathcal{B}_{\Sigma_{\la}}$ is a closed basis for $\Sigma_{\la}$. Let $\mathcal{C}(\cdot)$ be the desired Kuratowski closure operator. It suffices to show that 	$(\cdot)\ua=\mathcal{C}(\cdot)$. For any $T\preccurlyeq \Sigma_{\la}$, we obtain
\[ 
\mathcal{C}(T)=\bigwedge \left\{S\ua\in \mathcal{B}_{\Sigma_{\la}}\mid  T\preccurlyeq S\ua\right\}.
\]
Since $T\ua\in \mathcal{B}_{\Sigma_{\la}}$ and $T\preccurlyeq T\ua$, we deduce from  the equality $(T\ua)\ua=T\ua$ that $\mathcal{C}(T)\preccurlyeq T\ua$. To have the reverse inclusion, let $J\in T\ua$. Consider any $S\preccurlyeq\Sigma_{\la}$ with $T\preccurlyeq S\ua$. Then, we have 
\[
\mathcal{D}_S=\wedge S\ua\preccurlyeq \mathcal{D}_T.
\]
Since  $\mathcal{D}_T\preccurlyeq J$, as $J\in T\ua$, it follows that $\mathcal{D}_S\preccurlyeq J$. Thus, $J\in S\ua$.  In all then, we have shown that $J$ belongs to every set included in the intersection that describes $\mathcal{C}(T)$, and so $J\in \mathcal{C}(T)$, and so $T\ua\preccurlyeq \mathcal{C}(T)$. Consequently, $T\ua=\mathcal{C}(T)$, and  hence, $(\cdot)\ua$ is a Kuratowski closure operator.	
\end{proof}

\begin{remark}
From Theorem \ref{kclcb}, we can conclude that if $(-)\ua$ is not a Kuratowski closure operator, then $\mathcal{B}_{\Sigma_{\la}}$ cannot serve as a closed basis for a topology on a class $\Sigma_{\la}$ of ideals in $\la$. Therefore, our only option is to utilize $\mathcal{B}_{\Sigma_{\la}}$ as a closed subbasis to generate a topology on $\Sigma_{\la}$.	
\end{remark}

Let $\la$ be a quantale and $\Sigma_{\la}$ be a class of ideals in $\la$. A \emph{quantale topology} on $\Sigma_{\la}$, denoted by $\tau_{\la}$, is defined as the topology generated by the sets of the form (\ref{sbcs}) as a subbasis of closed sets. When the class $\Sigma_{\la}$ of ideals is equipped with this quantale topology, it is referred to as a \emph{quantale space}.
For the sake of brevity, a quantale  space $(\Sigma_{\la}, \tau_{\la})$ will be denoted simply as $\Sigma_{\la}$, employing the same notation as the underlying set of the topological space.

\begin{remark}
(1)
The concept we have referred to as quantale topology  is also known by other names in different contexts. In the context of general lattices, it is referred to as coarse lower topology (see \cite[A.8, p.\,589]{DST19}) or lower topology (see \cite[Definition O-5.4, p.\,43]{G.al.03}). In the context of rings, it is called as ideal topology (see \cite{DG22}).
	
(2) From Theorem \ref{kclcb}, it follows that a quantale topology coincides with a Zariski topology or a Stone topology when the underlying class of ideals is respectively prime or maximal. Furthermore, in a more general setting, a quantale topology can be viewed as a hull-kernel topology whenever the underlying set is the class of strongly irreducible ideals or any subclass thereof.
\end{remark}

\smallskip 
\section{Quasi-compactness}\label{qct}
\smallskip 

Recall that a topological space is said to be \emph{quasi-compact} if it satisfies the finite intersection property (or, equivalently, every open cover has a finite subcover), without assuming the space to be $T_2$. It is well-known that spectrum of prime ideals of a ring endowed with Zariski topology is quasi-compact. We first provide a similar result encompassing more classes of ideals of a quantale.  However, the proof of this result relies on the Alexander's subbasis theorem.

\begin{theorem}\label{csb}
Let $\la$ be a quantale. If  \,$\Sigma_{\la}$ is a class of ideals in $\la$ with the property that $\mathrm{Max}_\la\preccurlyeq \Sigma_{\la},$ then the   quantale space $\Sigma_{\la}$  quasi-compact. 
\end{theorem}

\begin{proof}
Suppose  $\{\mathcal K_{ \lambda}\}_{\lambda \in \Lambda}$ is a family of subbasic closed sets of $\Sigma_{\la}$  with $\bigwedge_{\lambda\in \Lambda}\mathcal K_{ \lambda}=\emptyset.$ This implies  that
$\mathcal K_{ \lambda}=X_{\lambda}^{\uparrow}$,  for all $\lambda \in \Lambda$ and $X_{\lambda}\preccurlyeq \Sigma_{\la}$  such that
\[\emptyset=\bigwedge_{\lambda \in \Lambda}X_{\lambda}^{\uparrow}=\left(\bigvee_{\lambda \in \Lambda} X_{\lambda}  \right)^{\uparrow}.\]
The second equality indeed holds. For
\[J\in \left(\bigvee_{\lambda \in \Lambda} X_{\lambda}  \right)^{\uparrow} \Leftrightarrow\mathcal{D}_{\bigvee_{\lambda \in \Lambda} X_{\lambda}} \preccurlyeq J\Leftrightarrow \forall \lambda \in \Lambda,\mathcal{D}_{ X_{\lambda}}\preccurlyeq J\Leftrightarrow \forall \lambda \in \Lambda, J\in X_{\lambda}\ua\Leftrightarrow J\in \bigwedge_{\lambda \in \Lambda}X_{\lambda}^{\uparrow}.\]
If $\bigvee_{\lambda \in \Lambda} X_{\lambda}  \neq \la,$ then  there exists $M\in \mathrm{Max}_\la$ such that $\bigvee_{\lambda \in \Lambda} X_{\lambda}  \preccurlyeq M.$ Moreover, \[ \mathcal{D}_{ X_{\lambda}} \preccurlyeq X_{\lambda}  \preccurlyeq \bigvee_{\lambda \in \Lambda} X_{\lambda}  \preccurlyeq M,\]
for all $\lambda \in \Lambda.$ Therefore, by hypothesis, $M\in X_{\lambda}^{\uparrow}=\mathcal K_{\lambda},$  for all $\lambda \in \Lambda$, a contradiction. Hence $\bigvee_{\lambda \in \Lambda} X_{\lambda}=\la,$ and thus $\top\in \bigvee_{\lambda \in \Lambda} X_{\lambda}.$ This implies the existence of a finite subset $\{\lambda_{\scriptscriptstyle 1}, \ldots, \lambda_{\scriptscriptstyle n}\}$ of $\Lambda$ such that $\top= \bigvee_{i=1}^n x_{\lambda_i},$ for some  $x_{\lambda_i}\in X_{\lambda_i}$. Hence  $\la= \bigvee_{i=1}^n X_{\lambda_i},$ which establishes the finite intersection property. Therefore, $\Sigma_{\la}$ is quasi-compact by Alexander's subbasis theorem.
\end{proof} 

\begin{corollary}\label{ciqc}
The quantale spaces of proper, maximal, prime, radical, primary, strongly irreducible, and irreducible ideals are all quasi-compact.
\end{corollary}

While Theorem \ref{csb} presents a sufficient condition for a quantale space to possess quasi-compactness, the subsequent proposition delivers a characterization of these spaces. However, before proceeding with the proposition, we need to establish a prerequisite result from the domain of topological spaces.

\begin{lemma}\label{upper-bound}
If $X$ is a quasi-compact, T$_0$-space, then every chain in $X$ has an upper bound. 
\end{lemma} 

\begin{proof}
Let $K\preccurlyeq X$ be a chain and let $\Omega:=\left\{\mathcal {C}(k)\mid k\in K \right\}$. Then $\mathcal G$ is clearly a chain of nonempty closed sets of $X$ and thus it has the finite intersection property. Since $X$ is quasi-compact, there is a point $z\in \bigwedge \Omega$ and, by definition, $k\preccurlyeq z$, for every $k\in K$, and this proves the result. 
\end{proof}

\begin{proposition}\label{cqc}
Let $\Sigma_{\la}$ be a class of ideals of a quantale $\la$. Suppose $\mathcal{M}(\Sigma_{\la})$ denotes the set of maximal elements of $\Sigma_{\la}$. Then the following are equivalent: 
\begin{enumerate}
\item $\Sigma_{\la}$ is quasi-compact;
		
\item for every $I\in\Sigma_{\la},$ there exists a $J\in\mathcal{M}(\Sigma_{\la})$ such that $I\preccurlyeq J$, and $\mathcal{M}(\Sigma_{\la})$ is quasi-compact.
\end{enumerate}
\end{proposition}

\begin{proof}
To show (2)$\Rightarrow$(1), let $\mathcal{O}$ be an open cover of $\Sigma_{\la}$. Then, $\mathcal{O}$ also covers $\mathcal{M}(\Sigma_{\la})$, and thus there is a finite subcover $\mathcal{O}'$ of $\mathcal{M}(\Sigma_{\la})$. Consider $I\in\Sigma_{\la}$. By hypothesis, there is $J\in\mathcal{M}(\Sigma_{\la})$ such that $I\preccurlyeq J$, and $O\in\mathcal{O}'$ such that $J\in O$. Then, $I\in O$, and thus $\mathcal{O}'$ is also a finite subcover for $\Sigma_{\la}$. Hence, $\Sigma_{\la}$ is quasi-compact.
Next we prove that (1)$\Rightarrow$(2). Suppose that $I\in\Sigma_{\la}$, and consider \[\Sigma_{\la}':=I^\uparrow\wedge \Sigma_{\la}.\] Then, $\Sigma_{\la}'$ is a closed subset of $\Sigma_{\la}$, and thus it is quasi-compact; by Lemma \ref{upper-bound}, every ascending chain in $\Sigma_{\la}'$ is bounded, and hence by Zorn's Lemma $\Sigma_{\la}'$ has maximal elements, \textit{i.e.}, $I\preccurlyeq J,$ for some $J\in\mathcal{M}(\Sigma_{\la})$.
Now, suppose $\mathcal{O}$ is an open cover of $\mathcal{M}(\Sigma_{\la})$. Then, $\mathcal{O}$ is also an open cover of $\Sigma_{\la}$, and thus it admits a finite subcover, which will be also a finite subcover of $\mathcal{M}(\Sigma_{\la})$. Thus, $\mathcal{M}(\Sigma_{\la})$ is quasi-compact.
\end{proof}

\begin{remark}
While  Theorem \ref{csb} can now be derived from Proposition \ref{cqc}, it is worth noting that utilizing Theorem \ref{csb} simplifies the verification of quasi-compactness for a quantale space (see Corollary \ref{ciqc}).
\end{remark}

In Theorem \ref{csb}, we have observed that the inclusion of all maximal ideals in a class  $\Sigma_{\la}$ of ideals
is a sufficient condition for the quantale space $\Sigma_{\la}$
to possess quasi-compactness. In the subsequent proposition, we establish that this inclusion is also a necessary condition for the class of finitely generated (proper) ideals.

\begin{proposition}
Let $\la$ be a quantale. If the quantale space $\Sigma_{\la}$ of finitely generated proper ideals is quasi-compact, then $\mx\preccurlyeq\Sigma_{\la}$.
\end{proposition}

\begin{proof}
Suppose that $M\in \mx$ with the property that $M$ is not finitely generated. We show that $\Sigma_{\la}$ is not quasi-compact. Let us consider the collection of closed subspaces: \[\mathcal K=\left\{\langle x\rangle^\uparrow\wedge \Sigma_{\la}\mid x\in M \right\}.\]  We claim that $\bigwedge \mathcal K=\emptyset$. If not, let $I\in \bigwedge \mathcal K$. Then $I$ is finitely generated and $M\preccurlyeq I$. Since $M$ is not a finitely generated ideal, we must have $I\neq M,$ and that implies $I=\la$, which contradicts the fact that $\Sigma_{\la}$ is proper. But clearly $\mathcal K$ has the finite intersection property, and hence, $\Sigma_{\la}$ is not quasi-compact.
\end{proof}

Let us now explore the characterization of Noetherian quantales in terms of quasi-compact quantale spaces. 

\begin{theorem}
A quantale $\la$ is Noetherian if and only if every quantale space $\Sigma_{\la}$ is quasi-compact.
\end{theorem}

\begin{proof}
To obtain the claim, it suffices to show that $\la$ is Noetherian if and only if $\id$ is a Noetherian quantale space. 
We show that every subclass $\Sigma_{\la}$ of $\id$ is quasi-compact. Consider a collection \[\mathcal L:=\left\{\Sigma_{\la}\wedge I_{\lambda}^\uparrow  \right\}_{\lambda\in \Lambda}\] of subbasic closed sets of $\Sigma_{\la}$ with the finite intersection property. Since $\la$ is Noetherian, the ideal $J:=\bigvee_{\lambda\in \Lambda}I_{\lambda}$ is finitely generated, and assume that $J=\langle \alpha_1,\ldots,\alpha_n\rangle$. For every $1\leqslant j\leqslant n$, there exists a finite subset $\Omega_j$ of $\Lambda$ such that $\alpha_j\in \bigvee_{\lambda\in \Omega_j}I_{\lambda}$. If $\Omega:=\bigcup\limits_{j=1}^n\Omega_j$, it immediately follows that $J=\bigvee_{\omega\in \Omega}I_{\omega}$. Hence we have
\[
\bigwedge\, \mathcal L=\Sigma_{\la}\wedge J^\uparrow=\Sigma_{\la}\wedge \left\{\bigvee_{\omega\in \Omega}I_{\omega}\right\}^\uparrow= \bigwedge_{\omega\in \Omega}\left(\Sigma_{\la}\wedge I_{\omega} ^\uparrow\right)\neq \emptyset,
\]
since $H$ is finite and $\mathcal L$ has the finite intersection property. Then the claim follows by Alexander's subbasis Theorem. 
Conversely, let $\id$ be a Noetherian space and that there exists an ideal $I\in \id$ that is not finitely generated. Then the subspace \[\Sigma_{\la}:=\{J\in \id\mid  J \;\text{is finitely generated and}\; J\preccurlyeq I \}\] is not quasi-compact. 
In fact, the collection of closed sets \[\mathcal L:=\left\{\Sigma_{\la}\wedge \langle i\rangle^\uparrow \right\}_{i\in I}\] of $\Sigma_{\la}$ clearly has the finite intersection property, but has empty intersection. 
\end{proof} 

In the proposition that follows, we present a sufficient condition for  the quasi-compactness of the quantale space of regular ideals in $\la$.

\begin{proposition}
If $\la$ is a local quantale, then the quantale space of regular ideals is quasi-compact.
\end{proposition}

\begin{proof}
Let $\la$ be a local quantale with a unique maximal ideal $M$. We will consider two cases:
\begin{enumerate}
\item If $M$ is regular, the claim follows from Theorem \ref{csb}. 

\item If $M$ consists of zero-divisors, then every regular element in $\la$ is invertible. Consequently, the set of regular ideals in $\la$ is empty. By definition, an empty set is also quasi-compact. \qedhere
\end{enumerate}
\end{proof}
\smallskip 

\section{Separation axioms}\label{spax}
\smallskip 

The purpose of this section is to investigate the lower separation axioms of quantale spaces, specifically $T_0$, $T_1$, and sobriety. Notably, sobriety will hold significant importance as it aids in characterizing quantale spaces that are spectral. We commence by presenting an assertive yet obvious separation property that holds for all quantale spaces.

\begin{proposition}\label{t0a}
Every quantale space is a $T_0$-space. 
\end{proposition} 

\begin{proof}
Let $I,$ $J\in \Sigma_{\la}$ such that $I^\uparrow=J^\uparrow.$  Since $J\in I^\uparrow,$ we have $I\preccurlyeq J$. Similarly, we obtain $J\preccurlyeq I$. Hence, $I=J$, and we have the claim.  	
\end{proof}

To proceed further with separation axioms, we first need the following lemma. This lemma, however, is also going to be useful in the context of connectedness.

\begin{lemma}\label{irrc}
Every subbasic closed set of the form $\left\{I^\uparrow\mid I\in I^\uparrow\right\}$ of a quantale space $\Sigma_{\la}$ is irreducible. 
\end{lemma} 

\begin{proof} 
In fact, we show more, \textit{i.e.}, $I^\uparrow=\mathcal{C}(I)$ for every $I$ such that $I\in I^\uparrow$. By definition of $\mathcal{C}(\cdot)$ and hypothesis,	 we have  $\mathcal{C}(I)\preccurlyeq I^\uparrow$. 
We obtain the other inclusion by considering the following two cases:

(1) $\mathcal{C}(I)= \Sigma_{\la}$. Then  we have the claim from the fact that
\[
\Sigma_{\la}=\mathcal{C}(I)\preccurlyeq I^\uparrow\preccurlyeq \Sigma_{\la}.
\]

(2) $\mathcal{C}(I)\neq \Sigma_{\la}$. Since $\mathcal{C}(I)$ is a closed set, there exists an index set $\Omega$ such that for each $\omega\in\Omega$, there is a positive integer $n_{\omega}$ and $I_{\omega 1},\dots, I_{\omega n_\omega}\in \id$ such that 
\[
\mathcal{C}(I)={\bigwedge_{\omega\in\Omega}}\left({\bigvee_{ i\,=1}^{ n_\omega}}I_{\omega i}^\uparrow\right).
\]	
Since 
$\mathcal{C}(I)\neq \Sigma_{\la}$, we may assume that  ${\bigvee_{ i\,=1}^{ n_\omega}}I_{\omega i}^\uparrow$ is non-empty for each $\omega$. Therefore,  $I\in   {\bigvee_{ i\,=1}^{ n_\omega}}I_{\omega i}^\uparrow$, for each $\omega$, and from which we can deduce that $I^\uparrow\preccurlyeq {\bigvee_{ i=1}^{ n_\omega}}I_{\omega i}^{\uparrow}$, that is, $I^\uparrow\preccurlyeq \mathcal{C}(I)$. 	
\end{proof}     

\begin{corollary}\label{spiir}
Every non-empty subbasic closed set of the quantale space $\idp$ is irreducible.
\end{corollary} 

\begin{theorem}\label{ZariskiT1}
A quantale space $\Sigma_{\la}$ is a $T_1$-space if and only if $M\in \Sigma_{\la}$ for all $M\in \mx$.
\end{theorem}

\begin{proof}
Suppose that $I\in \Sigma_{\la}$. Then $I\in I^{\uparrow}$, and so, by Corollary \ref{spiir}, $\mathcal{C}(I)=I^{\uparrow}$. Let $M$ be a maximal ideal of $\la$ such that $I\preccurlyeq M$. Then   \[M\in 	I^{\uparrow}=\mathcal{C}(I) = \{I\},\] where the second equality follows from the fact that $\Sigma_{\la}$ is a $T_{ 1}$-space. Therefore $M=I$, showing that $\Sigma_{\la}\preccurlyeq \mx$.  
Conversely,  $M^{\uparrow}=\{M\}$ holds for every $M\in \mx$, so that $M\in M^{\uparrow}$. Hence, $\mathcal{C}(M)=\{M\}$, by Corollary \ref{spiir} and this proves the desired claim.
\end{proof}

Next we consider the separation axiom of sobriety. Recall that a topological space $X$ is called \emph{sober} if every non-empty irreducible closed subset $\mathcal K$ of $X$ is of the form: $\mathcal K=\mathcal{C}(\{x\})$, the closure of an unique singleton set $\{x\}$. The following result characterizes sober quantale spaces.

\begin{theorem}
\label{sob}  
An quantale space $\Sigma_{\la}$ is sober if and only if
\begin{equation}\label{soc}
\bigwedge_{\substack{I\preccurlyeq J ,\, J\in \Sigma_{\la}}}  J^{\uparrow}\preccurlyeq I^{\uparrow},
\end{equation}
for every non-empty irreducible subbasic closed subset of $I^{\uparrow}$ of $\Sigma_{\la}$.
\end{theorem}

\begin{proof}
If $\Sigma_{\la}$ is a sober space and $I^{\uparrow}$ is a non-empty irreducible subbasic closed subset of $\Sigma_{\la}$, then  
$I^{\uparrow}=\mathcal{C}(\{J\})=J^{\uparrow}$
for some $J\in \Sigma_{\la}$, and we have $$\displaystyle J=\bigwedge_{\substack{I\preccurlyeq J,\, J\in \Sigma_{\la}}}  J^{\uparrow}\in \Sigma_{\la}.$$  
Conversely, suppose the condition (\ref{soc}) holds for every non-empty irreducible subset of $\Sigma_{\la}$. Let $\mathcal K$ be an irreducible closed subset of $\Sigma_{\la}$. Then $$\mathcal K=\bigwedge_{i\in \Omega}\left (  \bigvee_{j=1}^m J_{ji}^{\uparrow} \right)$$ for some ideals $J_{ji}$ of $\la$. Since $\mathcal K$ is irreducible, for every $i\in \Omega$ there exists an ideal $J_{ji}$ of $\la$ such that $\mathcal K\preccurlyeq J_{ji}^{\uparrow}\preccurlyeq \bigvee_{j=1}^mJ_{ji}^{\uparrow}$ and thus, if $Y= \bigvee_{i\in \Omega}J_{ji}$, then we have  $$\mathcal K=\bigwedge_{i\in \Omega}J_{ji}^{\uparrow}=Y^{\uparrow}=\left(\bigwedge_{\substack{Y\preccurlyeq L ,\, L\in \Sigma_{\la}}}  L \right)^{\uparrow}=\mathcal{C}\left({\bigwedge_{\substack{Y\preccurlyeq L,\, L\in \Sigma_{\la}}}  L}\right).$$ 
The uniqueness part of the claim follows from Proposition \ref{t0a}.
\end{proof}

\begin{corollary}\label{qss}
The quantale spaces of proper ideals, prime ideals, and strongly irreducible ideals are sober.
\end{corollary}

Hochster in \cite{Hoc69} has characterized spectral spaces  as spectra of prime ideals of rings endowed with Zariski topology. A topological space is called \emph{spectral} if it is quasi-compact, sober,  admitting a basis of quasi-compact open subspaces that is closed under finite intersections. Our next goal is to  prove the following result.

\begin{theorem}\label{tsqs}
Let $\la$ be a quantale. Then
the quantale space $\idp$ of proper ideals is spectral;
\end{theorem}

\begin{proof}
We prove the claim by applying the following lemma\footnote{for a proof of Lemma \ref{cso}, we refer to \cite[Lemma 2.2]{Gos23}}. 	

\begin{lemma}\label{cso}
A quasi-compact, sober, open subspace of a spectral space is spectral. 
\end{lemma}

The advantage of the above lemma is that we can avoid proving the existence of a compact open basis and closedness of open compact sets under finite intersections. 	

Now the quantale space $\id$ of all ideals is spectral follows from the fact that an algebraic lattice endowed with coarse lower topology\footnote{which is same as quantale topology in our context.} is a spectral space (see \cite[Theorem 4.2]{Pri94}).
Since the quantale space $\idp $ is a subspace of $\id$, according to Lemma \ref{cso}, we need to verify that  $\idp $ is quasi-compact, sober, and
is an open subspace of the quantale space $\id$. Now quasi-compactness and sobriety respectively follow from Corollary \ref{ciqc} and Corollary \ref{qss}. 
Therefore, what remains is  to show that  
the quantale space $\idp $ is open in $\id$.
Since $\la\in\id,$ by Lemma \ref{irrc}, $\la=\la^{\uparrow}=\mathcal{C}(\la),$ and therefore $\id \backslash\idp $ is closed, and that implies $\idp $ is open. 
\end{proof} 

\smallskip 
\section{Connectedness and continuity}\label{cac}
\smallskip 

Like Alexander subbasis theorem, there is no characterization of connectedness in terms of subbasic closed sets. Nevertheless, we wish to present a disconnectivity result of quantale spaces of a quantale that bears resemblance to the fact that if spectrum of prime ideals (of a commutative ring with identity) endowed with Zariski topology is disconnected, then the ring has a proper idempotent element (see \cite[\textsection 4.3, Corollary 2]{Bou72}). 

We say a closed subbasis $\mathcal{S}$ of a topological space $X$ \emph{strongly disconnects} $X$  if there exist two non-empty subsets $A,$ $B$ of $\mathcal{S}$ such that $X=A\vee  B$ and $A\wedge  B=\emptyset$. 
It is clear that if some closed subbasis strongly disconnects  a  topological space, then the space is disconnected. Also, if a space is disconnected, then some closed subbasis (for instance the collection of all its closed subspaces) strongly disconnects it. 

\begin{proposition}\label{pr1}  
Suppose that $\la$ is a quantale with  $\mathcal{J}(\la)=\theta$. Let $\Sigma_{\la}$ be a spectrum of $\la$ containing all maximal ideals in $\la.$ If the subbasis $\mathcal{S}$ of the quantale space $\Sigma_{\la}$ strongly  disconnects  $\Sigma_{\la}$, then $\la$ has a non-trivial idempotent element.
\end{proposition} 

\begin{proof}
Let $I$ and $J$ be ideals in $\la$ such that
(1) $
I^{\uparrow}\wedge  J^{\uparrow} =\emptyset,$
(2) $I^{\uparrow}\vee  J^{\uparrow} =\Sigma_{\la},$ and  
(3) $I^{\uparrow}\ne\emptyset,  J^{\uparrow} \neq \emptyset. $
Since $I^{\uparrow}\wedge  J^{\uparrow}=(I \vee J)^{\uparrow}$, we therefore have $(I\vee J)^{\uparrow}=\emptyset$ and hence $I\vee J =\la$ because $\Sigma_{\la}$ contains all maximal ideals in $\la$. On the other hand, 
\[
\Sigma_{\la}=I^{\uparrow}\vee  J^{\uparrow}\preccurlyeq (I\pr J)^{\uparrow},
\] which then implies that $I\pr J$ is contained in every maximal ideal in $\la$, and is therefore the zero ideal since $\la$ has zero $\mathcal{J}(\la)$. Note that the condition (3) implies that neither $I$ nor $J$ is the entire quantale $\la$. So the 
equality $I\vee J=\la$ furnishes non-zero elements $x\in I$ and $y\in J$ such that $x\vee y=\top$. Since  $x\pr y=\bot$ as $x\pr y\in I\pr J=\theta$, we therefore have
\[
x= x\pr(x\vee y)=x^{ 2}\vee (x\pr y) =x^2,
\]
showing that $x$ is a non-zero  idempotent element in $\la$. Since $I\ne \la$, and $x\ne \top$, and hence $x$ is a non-trivial idempotent element in $\la$.
\end{proof}

The next theorem gives a sufficient condition for a quantale space to be connected.

\begin{theorem}\label{conn}
If a spectrum $\Sigma_{\la}$ of a quantale $\la$ contains the zero ideal, then the quantale space $\Sigma_{\la}$ is connected. 
\end{theorem}

\begin{proof}
Since  $\Sigma_{\la}=\bot^{\uparrow}$ and irreducibility implies connectedness, the desired claim now follows from Corollary \ref{spiir}.
\end{proof}

\begin{corollary}
Quantale spaces of proper, finitely generated, principal ideals are connected. 
\end{corollary}

We now discuss about continuous maps between quantale spaces of quantales. Recall that a map $\phi\colon \la \to \la'$ from $\la$ to $\la'$ is called a \emph{quantale homomorphism} if 
\begin{enumerate}
\item  $x\preccurlyeq x'$ implies that $\phi(x)\preccurlyeq \phi(x');$	
	
\item  $\phi(x\vee x')=\phi(l)\vee \phi(l');$
	
\item  $\phi(x\wedge x')=\phi(x)\wedge \phi(x');$
	
\item  $\phi(x\pr  x')=\phi(x)\pr  \phi(x'),$
\end{enumerate}
for all $x,$ $x'\in \la$. 
Suppose that  $\phi\colon \la \to \la'$ is a quantale homomorphism.
If $J$ is an ideal in $\la'$, then the \emph{contraction of} $J$, denoted by $J^c$, is defined by $\phi\inv (J).$
Observe that although inverse image of an ideal under a quantale homomorphism is an ideal, but the same may not hold for an arbitrary class $\Sigma_{\la}$ of ideals. To resolve this problem, we need to impose that property on a class of ideal. 
We say a class $\Sigma_{\la}$ of ideals satisfies the \emph{contraction} property if for any quantale homomorphism $\phi\colon \la\to \la',$ the inverse image  $\phi\inv(I')$ is in $\Sigma_{\la}$, whenever $I'$ is in $\Sigma_{\la'}.$   
Since the sets $\{I^{\uparrow}\mid I\in \id\}$ only form a (closed) subbasis, all our arguments need to be at this level rather than just closed sets.

\begin{proposition}\label{conmap}
Let $\Sigma_{\la}$ be a spectrum satisfying the contraction property. Let $\phi\colon \la\to \la'$ be a quantale homomorphism  and $I'\in\Sigma_{\la'}.$ 
\begin{enumerate}
		
\item\label{contxr} The induced map $\phi_*\colon  \Sigma_{\la'}\to \Sigma_{\la}$ defined by  $\phi_*(I')=\phi\inv(I')$ is    continuous.
		
\item If $\phi$ is  surjective, then the quantale space $\Sigma_{\la'}$ is homeomorphic to the closed subspace $\mathrm{Ker}\phi^{\uparrow}$ of the quantale space $\Sigma_{\la}.$
		
\item\label{den} The subset  $\phi_*(\Sigma_{\la'})$ is dense in $\Sigma_{\la}$ if and only if $\mathrm{Ker}\phi\preccurlyeq \bigwedge_{I\in \Sigma_{\la}}I.$ 
\end{enumerate}
\end{proposition}

\begin{proof}      
To show (1), let $I\in \id$ and $I^{\uparrow}$ be a   subbasic closed set of the quantale space $\Sigma_{\la}.$ Then  
\begin{align*}
\phi_*\inv(I^{\uparrow}) &=\left\{ I'\in  \Sigma_{\la'}\mid \phi\inv(I')\in I^{\uparrow}\right\}\\&=\left\{I'\in \Sigma_{\la'}\mid \phi(I)\preccurlyeq I'\right\}=\left\langle\phi(I)\right\rangle^{\uparrow}, 
\end{align*} 
and hence the map $\phi_*$  continuous.  
		
(2) Observe that $\mathrm{Ker}\phi\preccurlyeq \phi\inv(I')$ follows from the fact that  $\theta\preccurlyeq I'$ for all $I'\in \Sigma_{\la'}.$ It can thus been seen that $\phi_*(I')\in \mathrm{Ker}\phi^{\uparrow},$ and hence $\mathrm{Im}\phi_*=\mathrm{Ker}\phi^{\uparrow}.$  
If $I'\in \Sigma_{\la'},$ then
\[\phi\left(\phi_*\left(I'\right)\right)=\phi\left(\phi\inv\left(I'\right)\right)=I'.\]
Thus $\phi_*$ is injective. To show that $\phi_*$ is closed, first we observe that for any   subbasic closed set  $I^{\uparrow}$ of  $\Sigma_{\la'}$, we have
\begin{align*}
\phi_*\left(I^{\uparrow}\right)&=  \phi\inv\left(I^{\uparrow}\right)\\&=\phi\inv\left\{ I'\in \Sigma_{\la'}\mid I\preccurlyeq   I'\right\}=\phi\inv(I)^{\uparrow}. 
\end{align*}
Now if $\mathcal K$ is a closed subset of $\Sigma_{\la'}$ and $\mathcal K=\bigwedge_{ \omega \in \Omega} \left(\bigvee_{ i \,= 1}^{ n_{\omega}} I_{ i\lambda}^{\uparrow}\right),$ then
\begin{align*}
\phi_*(\mathcal K)&=\phi\inv \left(\bigwedge_{ \omega \in \Omega} \left(\bigvee_{ i = 1}^{ n_{\omega}} I_{ i\lambda}^{\uparrow}\right)\right)\\&=\bigwedge_{ \omega \in \Omega}\, \bigvee_{ i = 1}^{n_{\omega}} \phi\inv\left(I_{ i\omega}\right)^{\uparrow}
\end{align*}
a closed subset of  $\Sigma_{\la}.$ Since by (\ref{contxr}), $\phi_*$ is continuous, we have the desired claim.
	
(3) First we wish to show: \[\mathcal{C}\left(\phi_*\left(I'^{\uparrow}\right)\right)=\phi\inv\left(I'\right)^{\uparrow},\] for all ideals $X'$ of $\la'$. For that, let $J\in \phi_*\left(I'^{\uparrow}\right).$ This implies $\phi\left(J\right)\in I'^{\uparrow},$ and that means $I'\preccurlyeq \phi(J).$ Therefore, $J\in \phi\inv\left(I'\right)^{\uparrow}.$ Since $\phi\inv\left(I'^{\uparrow}\right)=\phi\inv\left(I'\right)^{\uparrow}$, the other inclusion follows. If we take $I'$ as the trivial ideal $\theta'$ in $\la'$, the above identity reduces to 
$\mathcal{C}\left(\phi_*\left(\Sigma_{\la'}\right)\right)=\mathrm{Ker}\phi^{\uparrow},$ and hence  $\mathrm{Ker}\phi^{\uparrow}$ to be equal to $\Sigma_{\la}$ if and only if $\mathrm{Ker}\phi\preccurlyeq \bigwedge_{I\in \Sigma_{\la}}I.$ 
\end{proof}

\end{document}